\documentclass[12pt,a4paper]{amsart}

\usepackage{amssymb, amsthm}

\newtheorem{theorem}{Theorem}[section]
\newtheorem{lemma}[theorem]{Lemma}
\newtheorem{prop}[theorem]{Proposition}

\theoremstyle{remark}
\newtheorem{remark}[theorem]{Remark}
\newtheorem*{remark*}{Remark}
\newtheorem{defn}[theorem]{Definition}

\newtheorem*{notation*}{Notation}

\def\N{{\mathbb N}} \def\C{{\mathbb C}}\def\T{{\mathbb T}}

\newcommand{\clsp}{\overline{\operatorname{span}}}
\newcommand{\lsp}{\operatorname{span}}

\newcommand{\Aut}{\operatorname{Aut}}
\newcommand{\FE}{\operatorname{FE}}
\newcommand{\Ext}{\operatorname{Ext}}
\newcommand{\MCE}{\operatorname{MCE}}

\numberwithin{equation}{section}

\title[Aperiodicity and cofinality for finitely aligned higher-rank graphs.]
    {Aperiodicity and cofinality for finitely aligned higher-rank graphs}
\author{Peter Lewin}
\email{asims@uow.edu.au}
\author{Aidan Sims}
\email{pkl621@uow.edu.au}
\address{ School of Mathematics and Applied Statistics\\
Austin Keane Building (15)\\ University of
Wollongong\\ Wollongong NSW Australia.}
\thanks{This research was supported by the Australian Research Council.}
\date{May 6, 2009}

\begin{document}

\begin{abstract}
We introduce new formulations of aperiodicity and cofinality
for finitely aligned higher-rank graphs $\Lambda$, and prove
that $C^*(\Lambda)$ is simple if and only if $\Lambda$ is
aperiodic and cofinal. The main advantage of our versions of
aperiodicity and cofinality over existing ones is that ours are
stated in terms of finite paths. To prove our main result, we
first characterise each of aperiodicity and cofinality of
$\Lambda$ in terms of the ideal structure of $C^*(\Lambda)$. In
an appendix we show how our new cofinality condition simplifies
in a number of special cases which have been treated previously
in the literature; even in these settings, our results are new.
\end{abstract}

\maketitle

\section{Introduction}
From the groundbreaking work of Cuntz and Krieger~\cite{ck} the
theory of Cuntz-Krieger algebras has been generalised through
the efforts of many authors to include $C^*$-algebras of finite
directed graphs \cite{ew}, infinite directed graphs \cite{flr,
kprr}, infinite $\{0,1\}$-matrices \cite{el}, ultragraphs
\cite{tomforde}, topological graphs and quivers \cite{katsura,
mt}, and higher-rank graphs \cite{kp}, to name a few. In this
paper we focus on the $C^*$-algebras of finitely aligned
higher-rank graphs \cite{fmy, rsy2}. In generalisations of
Cuntz-Krieger algebras, simplicity is characterised by two
conditions on the graph, now known as aperiodicity and
cofinality. Cofinality is traditionally phrased in terms of
infinite paths, and in the setting of higher-rank graphs, the
same is true of aperiodicity. This is problematic because,
especially in higher-rank graphs, the infinite paths in
question can be difficult to identify and work with. In this
paper we introduce new formulations of aperiodicity and
cofinality --- which involve only finite paths --- for finitely
aligned higher-rank graphs $\Lambda$, and prove that
$C^*(\Lambda)$ is simple if and only if $\Lambda$ is aperiodic
and cofinal. This generalises the results of \cite{robsi1,
robsi2} to finitely aligned $k$-graphs.

A directed graph $E$ is a quadruple $(E^0, E^1, r, s)$ where
$E^0$ and $E^1$ are countable sets and $r$ and $s$ are maps
from $E^1$ to $E^0$. The elements of $E^0$ are called vertices
and the elements of $E^1$ are called edges. For each edge $e
\in E^1$ the vertex $r(e)$ is called the range of $e$ and
$s(e)$ is called the source of $e$. We visualise the vertices
as dots and each edge $e$ as an arrow from $s(e)$ to $r(e)$.

In 1980, Enomoto and Watatani~\cite{ew} associated a
$C^*$-algebra to each finite directed graph $E$ with no sources
as follows. Suppose $H$ is a Hilbert space. Then a
\emph{Cuntz-Krieger $E$-family} on $H$ consists of a set $\{
P_v : v \in E^0 \}$ of mutually orthogonal projections on $H$
and a set $\{ S_e : e \in E^1 \}$ of partial isometries on $H$
satisfying
\begin{enumerate}
\item $S_e^* S_e = P_{s(e)}$ for every $e \in E^1$; and
\item $P_v = \sum_{\{ e \in E^1 : r(e) = v \}} S_e S_e^*$
    for all $v \in E^0$.
\end{enumerate}
The two relations above are now known as the Cuntz-Krieger
relations. Enomoto and Watatani's definition was subsequently
generalised by various authors to infinite graphs \cite{bhrs,
bprs, flr, kpr, kprr}. In all cases a key justification of the
chosen Cuntz-Krieger relations is the so-called Cuntz-Krieger
uniqueness theorem. For directed graphs this theorem states
that if every loop in $E$ has an entrance, then any two
Cuntz-Krieger $E$-families consisting of nonzero partial
isometries generate isomorphic $C^*$-algebras.				

A higher-rank graph, or $k$-graph, is an analogue of a directed
graph in which paths have a degree in $\N^k$ rather than a
length in $\N$. These $k$-graphs and their $C^*$-algebras were
introduced by Kumjian and Pask \cite{kp} as graph-based models
for the higher-rank Cuntz-Krieger algebras studied by Robertson
and Steger \cite{rob-steg}. For technical reasons, Kumjian and
Pask only considered $k$-graphs in which each vertex receives
at least one and at most finitely many paths of any given
degree; such $k$-graphs are said to be row-finite with no
sources. Subsequently Raeburn, Sims and Yeend generalised the
theory of $k$-graph $C^*$-algebras to finitely aligned
$k$-graphs \cite{rsy1, rsy2}.

In recent years $k$-graph algebras have attracted a great deal
of attention. Exel~\cite{exel} realises higher rank graph
algebras as combinatorial algebras and recovers the underlying
path space from the algebra. Farthing, Muhly and
Yeend~\cite{fmy} provide inverse semigroup and groupoid models
for higher rank graph algebras while Katsoulis and
Kribs~\cite{kk} explore the relationship between Cuntz-Krieger
algebras of higher rank graphs and nonselfadjoint operator
algebras. Many other authors have contributed both to the
fundamental theory of $k$-graph algebras \cite{burgstaller,
pqr, sz} and to its applications \cite{dpy, dmmc, hms,
spielberg}. 				

In their seminal paper, Kumjian and Pask proved a
generalisation of the Cuntz-Krieger uniqueness theorem for
row-finite $k$-graphs with no sources~\cite[Theorem~4.6]{kp}.
Informed by the original groupoid model for graph
$C^*$-algebras \cite{kpr}, Kumjian and Pask observed that in a
directed graph the hypothesis that every loop has an entrance
ensures that each vertex receives at least one infinite path
which is not equal to any left-shift of itself. This was the
formulation which they generalised to the higher-rank graph
setting, and has become known as the aperiodicity condition. In
particular, the aperiodicity condition in Kumjian and Pask's
theorem is phrased in terms of infinite paths. In \cite{fmy,
rsy2} the Cuntz-Krieger uniqueness theorem was further
generalised to $k$-graphs which admit sources as well as
vertices which may receive an infinite number of edges of the
same degree. Each new generalisation has necessitated a new
notion of an infinite path and hence a new notion of
aperiodicity, so several different notions of aperiodicity have
now appeared in the literature. Moreover, as the class of
$k$-graphs considered broadens, the associated collection of
infinite paths becomes more complicated, so the corresponding
aperiodicity condition becomes harder to verify.

The Cuntz-Krieger uniqueness theorem can be reinterpreted as
the assertion that any nontrivial ideal must contain a vertex
projection. From the Cuntz-Krieger uniqueness theorem, it is
typically a short step to a sufficient condition for
simplicity. One identifies a cofinality condition which implies
that an ideal containing one vertex projection must contain all
the others, and hence must be the whole $C^*$-algebra. For
row-finite graphs with no sources, the appropriate condition is
that for any vertex $v$ and any infinite path $x$ in the graph,
there exists a path with range $v$ whose source lies on $x$.
This condition, like aperiodicity, becomes more and more
complicated for more general versions of the theory because the
appropriate notion of an infinite path becomes more involved.
In particular, for finitely aligned higher-rank graphs, the
appropriate notion of cofinality, recently identified by
Shotwell~\cite{shotwell}, is potentially quite difficult to
check in examples.

In this paper, we improve on previous formulations of both
aperiodicity and cofinality with equivalent conditions which
only involve finite paths. In particular, our new conditions
are more easily verified in practice than their predecessors.
Our main result is Theorem~\ref{main}: for a finitely aligned
$k$-graph $\Lambda$, $C^*(\Lambda)$ is simple if and only if
$\Lambda$ is aperiodic and cofinal. As well as involving only
finite paths, this is an improvement on
\cite[Proposition~8.5]{sims2} where only a sufficient condition
is established. Our presentation is as self-contained as
possible, and we have largely chosen to present direct proofs
rather than appeal to existing results elsewhere in the
literature.

Our characterisation of simplicity is not entirely new: we
discovered late in the course of this research that Shotwell
has recently proved that $C^*(\Lambda)$ is simple if and only
if $\Lambda$ has no local periodicity and is cofinal in the
sense of \cite{sims2}. His work is available as a
preprint~\cite{shotwell}. However, even where our results
converge with Shotwell's, our approach is quite different. For
example, though we use the same notion of no local periodicity
as Shotwell to prove that no local periodicity combined with
cofinality in $\Lambda$ implies that $C^*(\Lambda)$ is simple,
our proof is completely different from his. Shotwell's approach
is to show that no local periodicity is equivalent to
Condition~(A) of \cite{fmy} and then appeal to existing
results, while our approach is via a direct argument which does
not appeal to any heavy machinery. A similar comparison applies
to the two proofs (ours and Shotwell's) that cofinality is
necessary for simplicity. Moreover, as mentioned above, our
definitions of aperiodicity and cofinality involve only finite
paths, which make them easier to work with.

We begin by setting up the background and notation needed for
the rest of the paper in Section~\ref{prelim}. In
Section~\ref{ilp} we introduce our new definitions of
aperiodicity and cofinality, and state our main result,
Theorem~\ref{main}. As a first step to proving this main
theorem, we show that our aperiodicity condition is equivalent
to Shotwell's no local periodicity condition, and hence to two
other aperiodicity conditions used elsewhere in the literature.
In Section~\ref{is} we explore the $C^*$-algebraic consequences
of aperiodicity of a $k$-graph. In Section~\ref{prf} we show
that our notion of cofinality is equivalent to the definition
of cofinality in~\cite{shotwell, sims2} and then explore the
relationship between cofinality of $\Lambda$ and the structure
of $C^*(\Lambda)$; we conclude Section~\ref{prf} with the proof
of our main result. We have also included an appendix in which
we indicate how our new cofinality condition simplifies in a
number of special cases; it is new even in these contexts.

\subsection*{Acknowledgements}
We record our gratitude to Jacob Shotwell, for keeping us
abreast of his research and supplying us with a preprint.

\section{Preliminaries}\label{prelim}

\subsection{$k$-graphs}

We regard $\N^k$ as a semigroup under addition, and use $e_i$
to denote the $i$th generator. For $m, n \in \N^k$ we write
$m_i, n_i$ for the $i$th coordinates of $m$ and $n$ and $m \vee
n$ for their coordinatewise maximum and $m \wedge n$ for their
coordinatewise minimum.

A $k$-graph $(\Lambda, d)$ is a category $\Lambda$ endowed with
a functor $d : \Lambda \rightarrow \N^k$ satisfying the
\emph{factorisation property}: for every $\lambda \in \Lambda$
and $m, n \in \N^k$ satisfying $d(\lambda) = m + n$ there exist
unique $\mu, \nu \in \Lambda$ such that $d(\mu) = m$, $d(\nu) =
n$ and $\lambda = \mu \nu$. We read $d(\lambda)$ as the
\emph{degree of $\lambda$} and think of $d$ as a generalised
length function. For $n \in \N^k$ we write $\Lambda^n$ for $\{
\lambda \in \Lambda : d(\lambda) = n \}$.

The factorisation property implies that for each $\lambda \in
\Lambda$ there exist unique elements $r(\lambda), s(\lambda)
\in \Lambda^0$ such that $\lambda = r(\lambda) \lambda =
\lambda s(\lambda)$, and also that $\Lambda^0 = \{ r(\lambda) :
\lambda \in \Lambda \} = \{ s(\lambda) : \lambda \in \Lambda
\}$. We call elements of $\Lambda^0$ vertices, and the
functions $r, s$ the \emph{range} and \emph{source} maps. Given
$k$-graphs $\Lambda$ and $\Gamma$, a $k$-graph morphism $\phi$
from $\Lambda$ to $\Gamma$ is a functor which preserves the
degree map.

For $\lambda \in \Lambda$ and $m \leq n \leq d(\lambda)$ we
define $\lambda(m, n)$ to be the unique path in $\Lambda^{n -
m}$ obtained from the factorisation property such that $\lambda
= \lambda' \big( \lambda (m, n) \big) \lambda''$ for some
$\lambda'\in \Lambda^m$ and $\lambda'' \in \Lambda^{d(\lambda)
- n}$. Unlike in \cite{rsy1, rsy2}, we do \textbf{not} write
$\lambda(0, n)$ to mean $\lambda \big( 0, n \wedge d(\lambda)
\big)$; so $\lambda(0, n)$ is undefined if $n \nleq
d(\lambda)$, and $d \big( \lambda(p, q) \big)$ is always equal
to $q - p$ when $\lambda(p, q)$ makes sense. For $n \in \N^k$
define
\[
\Lambda^{\leq n} :=
    \{ \lambda \in \Lambda : d(\lambda) \leq n, \text{ and } d(\lambda)_i < n_i \Rightarrow s(\lambda)\Lambda^{e_i} = \emptyset \}.
\]
For $v \in \Lambda^0$ and $X \subset \Lambda$, define $vX = \{
\lambda \in X : r(\lambda) = v \}$.
		
\subsection{Cuntz-Krieger families and $k$-graph $C^*$-algebras}

In this subsection we indicate how to associate a $C^*$-algebra
to a higher-rank graph.

Fix a $k$-graph $\Lambda$. For $\mu, \nu \in \Lambda$, we write
\begin{equation}
\MCE(\mu, \nu) := \{ \lambda \in \Lambda : d(\lambda) = d(\mu) \vee d(\nu), \lambda \big( 0, d(\mu) \big)
    = \mu, \lambda \big( 0, d(\nu) \big) = \nu \} \label{MCE}
\end{equation}
for the collection of \emph{minimal common extensions} of $\mu$
and $\nu$. We say that $\Lambda$ is \emph{finitely aligned} if
$|{\MCE(\mu, \nu)}| < \infty$ for all $\mu, \nu \in \Lambda$.

Let $\Lambda$ be a $k$-graph and fix $v \in \Lambda^0$ and $E
\subset v \Lambda$. We say that $E$ is \emph{exhaustive} if for
each $\mu \in v\Lambda$ there exists $\lambda \in E$ such that
$\MCE(\mu, \lambda) \neq \emptyset$. If $|E| < \infty$ we say
$E$ is \emph{finite exhaustive}. Define $\FE(\Lambda)$ to be
the set of finite exhaustive sets of $\Lambda$ and for each $v
\in \Lambda^0$, define $v \FE(\Lambda)$ to be the set of finite
exhaustive sets whose elements all have range $v$.

\begin{defn}[{\cite[Definition~2.5]{rsy2}}]\label{ckfam}
Let $\Lambda$ be a finitely aligned $k$-graph. A
\emph{Cuntz-Krieger $\Lambda$-family} is a set $\{ t_\lambda :
\lambda \in \Lambda \}$ of partial isometries satisfying:
\begin{enumerate}
\item[(CK1)] $\{ t_v : v \in \Lambda^0 \}$ is a collection
    of mutually orthogonal projections;
\item[(CK2)] $t_\mu t_\nu = t_{\mu \nu}$ whenever $s(\mu) =
    r(\nu)$;
\item[(CK3)] $t_\mu^* t_\nu = \sum_{\mu \alpha = \nu \beta
    \in \MCE(\mu, \nu)} t_\alpha t_\beta^*$ for every $\mu,
    \nu \in \Lambda$ and;
\item[(CK4)] $\prod_{\lambda \in E} (t_v - t_\lambda
    t_\lambda^*) = 0$ for every $v \in \Lambda^0$ and $E
    \in v \FE(\Lambda)$.
\end{enumerate}
\end{defn}

Given a finitely aligned $k$-graph $\Lambda$ there exists a
$C^*$-algebra $C^*(\Lambda)$ generated by a Cuntz-Krieger
$\Lambda$-family $\{ s_\lambda : \lambda \in \Lambda \}$ which
is universal in the following sense: given any other
Cuntz-Krieger $\Lambda$-family $\{ t_\lambda : \lambda \in
\Lambda \}$ there exists a unique homomorphism $\pi_t$ such
that $\pi_t (s_\lambda) = t_\lambda$ for every $\lambda \in
\Lambda$.

\begin{lemma}[{\cite[Lemma 2.7]{rsy2}}]\label{2.7}
Let $\Lambda$ be a finitely aligned $k$-graph and let $\{
t_\lambda : \lambda \in \Lambda \}$ be a Cuntz-Krieger
$\Lambda$-family. Then
\begin{enumerate}
\item[(i)] $t_\mu t_\mu^* t_\nu t_\nu^* = \sum_{\lambda \in
    \MCE(\mu, \nu)} t_\lambda t_\lambda^*$ for all $\mu,
    \nu \in \Lambda$. In particular, $\{ t_\lambda
    t_\lambda^* : \lambda \in \Lambda \}$ is a family of
    commuting projections.
\item[(ii)] For $\mu, \nu \in \Lambda^{\leq n}$, we have
    $t_\mu^* t_\nu = \delta_{\mu, \nu} t_{s(\mu)}$.
\item[(iii)] If $E \subset v \Lambda^{\leq n}$ is finite,
    then $t_v \geq \sum_{\lambda \in E} t_\lambda
    t_\lambda^*$.
\item[(iv)] $C^*(\{ t_\lambda : \lambda \in \Lambda \}) =
    \clsp\{ t_\mu t_\nu^* : \mu, \nu \in \Lambda \} =
    \clsp\{ t_\mu t_\nu^* : \mu, \nu \in \Lambda, s(\mu) =
    s(\nu) \}$.
\end{enumerate}
\end{lemma} 				

We have written~(CK3) and Lemma~\ref{2.7}(i) in terms of
$\MCE(\mu, \nu)$, whereas they are rendered in \cite{rsy2} in
terms of a different set, denoted $\Lambda^{\min}(\mu,\nu)$.
The two definitions are equivalent because, for fixed $\mu, \nu
\in \Lambda$ the map $(\alpha, \beta) \mapsto \mu \alpha$ is a
bijection between $\Lambda^{\min} (\mu, \nu)$ and $\MCE(\mu,
\nu)$. We will avoid reference to $\Lambda^{\min}$ in this
paper to reduce the amount of notation required as much as
possible.

\subsection{The $\partial \Lambda$ representation}

We construct for each finitely aligned $k$-graph $\Lambda$ a
Cuntz-Krieger $\Lambda$-family consisting of nonzero partial
isometries. We begin by making sense of paths of infinite
degree in a $k$-graph.

Let $k \in \{1,\ldots,k\}$ and $m \in (\N \cup \{ \infty
\})^k$. We define a $k$-graph $(\Omega_{k, m}, d)$ as follows:
for each $p \in \N^k$, the morphisms of degree $p$ are
\begin{align*}
    \Omega_{k, m}^p = \{ (q, q + p) : q \in \N^k, q + p \leq m \},
\end{align*}
and we define $r(p, q) := (p, p)$, $s(p, q) := (q, q)$, $d(p,
q) := q - p$, and $(p,q)(q,r) = (p,r)$. By convention, we
denote a vertex $(q, q)$ of $\Omega_{k, m}$ just by $q$.

For the next definition, recall that a $k$-graph morphism is a
degree-preserving functor between $k$-graphs. Given a $k$-graph
morphism $x : \Omega_{k,m} \to \Lambda$, we denote $x(0)$ by
$r(x)$, and $m$ by $d(x)$.

\begin{defn}[{\cite[Definition 5.10]{fmy}}]\label{bdp}
Fix a finitely aligned $k$-graph $\Lambda$. We denote by
$\partial \Lambda$ the set
\[
\bigcup_{m \in (\N \cup \{\infty\})^k} \big\{x : \Omega_{k,m} \to \Lambda \mid{}
    \parbox[t]{0.65\textwidth}{\noindent$x$ is a $k$-graph morphism, and for all
                                $n \in \mathbb{N}^k$ with $n \leq d(x)$ and all
                                $E \in x(n)\FE(\Lambda)$, there exists
                                $p \leq d(x) - n$ such that $x(n, n + p) \in E$
                                \big\}.}
\]
For $v \in \Lambda^0$, we write $v(\partial \Lambda)$ for $\{ x
\in \partial \Lambda : r(x) = v \}$.
\end{defn}

In \cite{fmy}, the elements of $\partial \Lambda$ were referred
to as \emph{boundary paths}. However, the same term has been
used elsewhere in the $k$-graph literature \cite{farthing,
rsy1, sims1} with different meanings. In order to avoid
confusion we refrain from using this term at all.

We will use the following lemma to construct a concrete
Cuntz-Krieger $\Lambda$-family below. This lemma is not new ---
indeed it is proved in greater generality in the thesis
\cite[Lemma~4.3.3]{ASPhD} --- but we nevertheless provide a
short proof here for ease of reference. To state the lemma,
recall that if $x : \Omega_{k,m} \rightarrow \Lambda$ is a
graph morphism, then: (1)~for each $n \in \N^k$ with $n \le m$,
there is a graph morphism $\sigma^n(x) : \Omega_{k, m-n} \to
\Lambda$ determined by $\sigma^n(x)(p,q) := x(n+p, n+q)$;
and~(2) for each $\lambda \in \Lambda r(x)$, there is a unique
graph morphism $\lambda x : \Omega_{k,m+d(\lambda)} \to
\Lambda$ such that $(\lambda x)(0, d(\lambda)) = \lambda$ and
$\sigma^{d(\lambda)}(\lambda x) = x$.

\begin{lemma}\label{addsub}
Let $(\Lambda, d)$ be a finitely aligned $k$-graph and let $x
\in \partial \Lambda$.
\begin{enumerate}
\item If $m \in \N^k$ and $m \leq d(x)$, then $\sigma^m (x)
    \in \partial \Lambda$.
\item If $\lambda \in \Lambda r(x)$, then $\lambda x \in
    \partial \Lambda$.
\end{enumerate}
\end{lemma}

To prove this lemma we must recall a definition and another
lemma from~\cite{rsy2}. We also use this definition and lemma
again in section~\ref{ilp}. Suppose $\lambda \in \Lambda$ and
$E \subseteq r(\lambda) \Lambda$, write $\Ext(\lambda, E)$ for
the set
\[
\bigcup_{\mu \in E} \{ \nu \big( d(\lambda), d(\lambda) \vee d(\mu) \big) :
        \nu \in \MCE(\lambda,\mu) \}. \label{ext}
\]

\begin{lemma}[{\cite[Lemma~C.5]{rsy2}}]\label{c5}
Let $(\Lambda, d)$ be a finitely aligned $k$-graph, let $v \in
\Lambda^0$, $\lambda \in v\Lambda$ and suppose $E \in
vFE(\Lambda)$. Then $\Ext(\lambda, E) \in s(\lambda)
FE(\Lambda)$.
\end{lemma}
\begin{proof}[Proof of Lemma~\ref{addsub}]
(1) Fix $n \in \N^k$ such that $n \leq d\big( \sigma^m (x)
\big)$ and fix $E \in \sigma^m(x)(n) \FE(\Lambda)$. Then $m + n
\leq d(x)$ and since $x \in \partial \Lambda$ there exists $p
\leq d(x) - (m + n)$ such that $x(m + n, m + n + p) \in E$.
That is, $p \leq d \big( \sigma (x) \big) - n$ and $\big(
\sigma ^m (x) \big) (n, n + p) \in E$.

(2) Fix $E \in (\lambda x)(n)\FE(\Lambda)$. Let $\lambda' =
(\lambda x) \big(n, n \vee d(\lambda) \big)$. By
Lemma~\ref{c5}, $\Ext(\lambda', E) \in x \big( (n \vee
d(\lambda)) - n \big) \FE(\Lambda)$, so there exists $p \leq
d(x) - \big( (n \vee d(\lambda)) - d(\lambda) \big)$ such that
$\alpha = (\lambda x) \big( n \vee d(\lambda), (n \vee
d(\lambda)) + p \big) = x \big((n \vee d(\lambda)) -
d(\lambda), (n \vee d(\lambda)) - d(\lambda) + p \big)$ belongs
to $\Ext(\lambda', E)$, say $\lambda' \alpha = \mu \beta$ where
$\mu \in E$. But now $q = d(\mu)$ satisfies
\begin{flalign*}
&&&(\lambda x) (n, n + q) = (\lambda' \alpha) (0, q) = (\mu
\beta) (0, q) = \mu \in E.&\qedhere
\end{flalign*}
\end{proof}

\begin{defn}[The $\partial \Lambda$ Representation]\label{bdpr}
Let $(\Lambda, d)$ be a finitely aligned $k$-graph and let $\{
\zeta_x : x \in \partial \Lambda \}$ denote the standard
orthonormal basis for $\ell^2 (\partial \Lambda)$. For $\lambda
\in \Lambda$, define
\begin{align*}
S_\lambda \zeta_x =
    \begin{cases}
        \zeta_{\lambda x} & \mathrm{if} ~ s(\lambda) = r(x), \\
        0 & \mathrm{otherwise.}
    \end{cases}
\end{align*}
Then $\{ S_\lambda : \lambda \in \Lambda \}$ is a Cuntz-Krieger
$\Lambda$-family in $\mathcal{B} \big( \ell^2(\partial \Lambda)
\big)$ called the \emph{$\partial \Lambda$ representation} (see
\cite[Lemma~4.6]{sims1}), and for each $\lambda \in \Lambda$,
\[
S^*_\lambda \zeta_x =
    \begin{cases}
        \zeta_{\sigma^{d(\lambda)}(x)} & \text{if $x(0,d(\lambda)) = \lambda$} \\
        0 & \text{otherwise.} 
    \end{cases}
\]
\end{defn}

\begin{remark}\label{0rep}
Fix a finitely aligned $k$-graph $\Lambda$. Lemma~5.15
of~\cite{fmy} implies that each $S_v \in
\ell^2(\partial\Lambda)$ is nonzero. The universal property of
$C^* (\Lambda)$ then implies that the universal generating
partial isometries in $C^*(\Lambda)$ are all nonzero.
\end{remark}

\section{Aperiodicity, cofinality, and the main theorem}\label{ilp}
In this section we introduce our new formulations of
aperiodicity and cofinality, and state our main result,
Theorem~\ref{main}.

\begin{defn}\label{AP}
We say that a $k$-graph $\Lambda$ is \emph{aperiodic} if for
every pair of distinct paths $\alpha, \beta \in \Lambda$ with
$s(\alpha) = s(\beta)$ there exists $\tau \in s(\alpha)
\Lambda$ such that $\MCE(\alpha \tau, \beta \tau) = \emptyset$.
\end{defn}

\begin{remark}\label{0min}
To check that $\Lambda$ is aperiodic, it suffices to show that
for every distinct pair $\mu, \nu$ such that $s(\mu) = s(\nu)$,
$r(\mu) = r(\nu)$ and $d(\mu) \wedge d(\nu) = 0$, there exists
$\tau \in s(\mu)\Lambda$ such that $\MCE(\mu \tau, \nu \tau) =
\emptyset$. For suppose that this is indeed the case, and fix
distinct $\alpha,\beta \in \Lambda$ with $s(\alpha) =
s(\beta)$. Let $m := d(\alpha) \wedge d(\beta)$, let $\mu :=
\alpha \big( m, d(\alpha) \big)$, and let $\nu := \beta \big(
m, d(\beta) \big)$. If $\alpha(0, m) \not= \beta(0, m)$, then
$\tau = s(\alpha)$ satisfies $\MCE(\alpha \tau, \beta \tau) =
\emptyset$. On the other hand, if $\alpha(0, m) = \beta(0, m)$,
then that $\alpha \not= \beta$ forces $\mu \not= \nu$. Hence by
assumption there exists $\tau \in s(\mu)\Lambda$ such that
$\MCE(\mu \tau, \nu \tau) = \emptyset$. Thus $\MCE(\alpha \tau,
\beta \tau) = \{ \alpha(0,m) \rho : \rho \in \MCE(\mu \tau, \nu
\tau) \} = \emptyset$ as required.
\end{remark}

\begin{defn}
Suppose $\Lambda$ is a $k$-graph. We say that $\Lambda$ is
\emph{cofinal} if for every $v, w \in \Lambda^0$ there exists
$E \in w\FE(\Lambda)$ such that $v \Lambda s(\alpha) \neq
\emptyset$ for every $\alpha \in E$.
\end{defn}

\begin{theorem}\label{main}
Suppose $\Lambda$ is a finitely aligned $k$-graph. Then
$C^*(\Lambda)$ is simple if and only if $\Lambda$ is aperiodic
and cofinal.
\end{theorem}

The proof of Theorem~\ref{main} occupies the rest of the paper.
We begin by establishing the equivalence of aperiodicity with
earlier conditions appearing in the literature. To this end we
recall the notion of local periodicity for finitely aligned
$k$-graphs. This was introduced by Shotwell in
(\cite[Definition~3.1 and Remarks~3.4]{shotwell}).

\begin{defn}\label{lp}
Let $(\Lambda, d)$ be a finitely aligned $k$-graph and suppose
$m, n \in \N^k$ are distinct. We say $\Lambda$ has \emph{local
periodicity $m,n$ at $v$} if for every $x \in v (\partial
\Lambda)$ we have:
\begin{enumerate}
\item[(LP1)] $m \vee n \leq d(x)$; and
\item[(LP2)] $\sigma^{m} (x) = \sigma^{n} (x)$.
\end{enumerate}
We say $\Lambda$ has \emph{no local periodicity} if $\Lambda$
does not have local periodicity $m,n$ at any $v \in \Lambda^0$
for any distinct $m, n \in \mathbb{N}^k$. That is, $\Lambda$
has no local periodicity if for every $v \in \Lambda^0$, and
every distinct $m, n \in \N^k$ either:
\begin{enumerate}
\item[(NLP1)] there exists $x \in v (\partial \Lambda)$
    such that $d(x) \ngeq m \vee n$; or
\item[(NLP2)] $d(x) \geq m \vee n$ for every $x \in v
    (\partial \Lambda)$, and there exists $y \in v(\partial
    \Lambda)$ such that $\sigma^m (y) \neq \sigma^n (y)$.
\end{enumerate}
\end{defn}

Recall from \cite{fmy} that $\Lambda$ is said to satisfy
condition~(A) if for each $v \in \Lambda^0$ there exists $x \in
v (\partial \Lambda)$ such that $\sigma^m (x) = \sigma^n (x)$
implies $m = n$ for all $m, n \leq d(x)$. Also recall from
\cite{rsy2} that $\Lambda$ is said to satisfy condition~(B) if
for each $v \in \Lambda^0$ there exists $x \in v \Lambda^{\leq
\infty}$ such that $\lambda, \mu \in \Lambda$ and $\lambda \neq
\mu$ imply $\lambda x \neq \mu x$.

\begin{prop}\label{teqnlp}
Suppose $\Lambda$ is a finitely aligned $k$-graph. Then the
following are equivalent:
\begin{enumerate}
\item $\Lambda$ is aperiodic,
\item $\Lambda$ has no local periodicity,
\item $\Lambda$ satisfies condition~(A) of
    \cite[Theorem~7.1]{fmy}
\item $\Lambda$ satisfies condition~(B) of
    \cite[Theorem~4.5]{rsy2}
\end{enumerate}
\end{prop}
\begin{proof}
Shotwell uses condition~(A) as his definition of aperiodicity
and establishes $(2) \Leftrightarrow (3) \Leftrightarrow (4)$
in \cite[Proposition~3.10]{shotwell}, so it suffices to show
$(1) \Leftrightarrow (2)$.

$(1) \Rightarrow (2)$. Suppose $\Lambda$ is aperiodic. Fix $v
\in \Lambda^0$ and distinct $m, n \in \N^k$. If there exists $x
\in v(\partial \Lambda)$ such that $d(x) \ngeq m \vee n$ then
(NLP1) holds and we are done. So we may suppose that $d(x) \geq
m \vee n$ for every $x \in v(\partial \Lambda)$. Since $v
(\partial \Lambda)$ is nonempty we have $v \Lambda^{m \vee n}
\neq \emptyset$. Fix $\lambda \in v \Lambda^{m \vee n}$ and
write $\lambda = \mu \alpha = \nu \beta$ where $d(\mu) = m$ and
$d(\nu) = n$. Note that this implies $d(\mu) + d(\alpha) =
d(\nu) + d(\beta)$ so $d(\alpha) - d(\beta) = d(\nu) - d(\mu)$
and in particular $d(\alpha) \neq d(\beta)$. Since $\Lambda$ is
aperiodic, there exists $\tau$ such that $\MCE(\alpha \tau,
\beta \tau) = \emptyset$. Fix $z \in s(\tau) (\partial
\Lambda)$ and let $y = \lambda \tau z \in v(\partial \Lambda)$.
Then $\sigma^m (y) = \alpha \tau z$ and $\sigma^n (y) = \beta
\tau z$. If these were equal then $\big( \sigma^m (y) \big)
\big( 0, d(\alpha \tau) \vee d(\beta \tau) \big)$ would belong
to $\MCE(\alpha \tau, \beta \tau)$, contradicting our choice of
$\tau$. So we must have $\sigma^n (y) \neq \sigma^m (y)$.

$(2) \Rightarrow (1)$. Suppose $\Lambda$ has no local
periodicity and fix distinct $\alpha, \beta \in \Lambda$ with
$s(\alpha) = s(\beta) = v$ and $d(\alpha) \neq d(\beta)$. By
Remark~\ref{0min} we may suppose that $r(\alpha) = r(\beta)$
and $d(\alpha) \wedge d(\beta) = 0$. Let $D := d(\alpha) \vee
d(\beta) = d(\alpha) + d(\beta)$. Then $D - d(\alpha) =
d(\beta)$ and $D - d(\beta) = d(\alpha)$. We consider two
cases: for $m = d(\beta)$ and $n = d(\alpha)$ either
\begin{enumerate}
\item (NLP1) holds.; or
\item (NLP2) holds
\end{enumerate}
We consider case (2) first because it is simpler.

\textbf{Case (2).} Suppose (NLP2) holds. That is, for every $x
\in v (\partial \Lambda)$ we have $d(x) \geq D$ and there
exists $y \in v (\partial \Lambda)$ such that
$\sigma^{d(\beta)} (y) \neq \sigma^{d(\alpha)} (y)$. Then there
exists $M_{\alpha, \beta} \leq d(y) - (d(\beta) \vee
d(\alpha))$ such that
\begin{align}
    \big( \sigma^{d(\beta)} (y) \big) (0, M_{\alpha, \beta}) &\neq \big( \sigma^{d(\alpha)} (y) \big) (0, M_{\alpha, \beta}) \nonumber
\intertext{or equivalently,}
    y(d(\beta), d(\beta) + M_{\alpha, \beta}) &\neq y(d(\alpha), d(\alpha) + M_{\alpha, \beta}). \label{A}
\end{align}

Let $\tau = y \big( 0, D + M_{\alpha, \beta})$. Then $d(\beta)
+ M_{\alpha, \beta} \leq d(\tau)$ and $d(\alpha) + M_{\alpha,
\beta} \leq d(\tau)$. Moreover,
\[
(\alpha \tau) (D, D + M_{\alpha, \beta})
    = \tau \big( D - d(\alpha), D - d(\alpha) + M_{\alpha, \beta} \big)
    = y \big( d(\beta), d(\beta) + M_{\alpha, \beta} \big)
\]
and similarly,
\[
(\beta \tau) (D, D + M_{\alpha, \beta}) = y\big( d(\alpha), d(\alpha) + M_{\alpha, \beta} \big).
\]
It follows from~\eqref{A} that $\MCE(\alpha \tau, \beta \tau) =
\emptyset$.

\textbf{Case (1).} Now suppose (NLP1) holds. That is, there
exists $x \in v (\partial \Lambda)$ such that $d(x) \ngeq D$.
We consider two subcases:
\begin{itemize}
\item[(i)] for every $\lambda \in v \Lambda$ there exists
    $\lambda' \in s(\lambda)\Lambda$ such that $d(\lambda
    \lambda') \geq D$; or
\item[(ii)] there exists $\lambda \in \Lambda$ such that
    $s(\lambda)\Lambda^{D - d(\lambda)} = \emptyset$.
\end{itemize}

\textbf{Case (1)(i).} We claim there exist $\tau_f \in v
\Lambda$ and $i \in \{ 1,\ldots,k \}$ such that $d(\alpha)_i
\neq d(\beta)_i$ and $|s(\tau_f) \Lambda^{e_i}| = \infty$. Let
$p = D \wedge d(x)$, $\lambda = x(0, p)$ and $G = s(\lambda)
\Lambda^{D - p}$. Then (i) implies that $G$ is exhaustive.
Moreover, since $d(x) \ngeq D$, there is no $q \leq d(x) - n$
such that $x(p, p + q) \in G$. Since $x \in \partial \Lambda$
it follows that $G \notin \FE(\Lambda)$ and hence that $|G| =
\infty$.

Since $D - p < \infty$ there exists $i \in \{1,\ldots,k\}$ and
a smallest $a < (D - p)_i \in \N$ such that $|\{ \mu \big( a
e_i, (a + 1) e_i \big) : \mu \in G \}| = \infty$. Since $a$ is
the smallest element of $\N$ with the given property we have
$|\{ \mu(0, a e_i) : \mu \in G \}| < \infty$ and there exists
$\lambda' \in G$ such that $|\lambda'(a e_i) \Lambda^{e_i}| =
\infty$. Let $\tau_f = \lambda \lambda'(0, a e_i)$.

We now show that $d(\alpha)_i \neq d(\beta)_i$. Since $\lambda'
\in G$ we have the inequality $0 < (a + 1) \leq d(\lambda')_i =
( D - p)_i$, which implies $D_i > 0$. Since $D = d(\alpha) \vee
d(\beta)$ and since $d(\alpha) \wedge d(\beta) = 0$ by
assumption, it follows that $d(\alpha)_i \neq d(\beta)_i$. This
establishes the claim.

Since $\Lambda$ is finitely aligned, $|\MCE(\alpha \tau_f,
\beta \tau_f)| < \infty$ and since $d(\alpha \tau_f)_i \neq
d(\beta \tau_f)_i$ we deduce that $d(\psi)_i > \max\{ d(\alpha
\tau_f)_i, d(\beta \tau_f)_i \}$ for every $\psi \in
\MCE(\alpha \tau_f, \beta \tau_f)$. However $|s(\tau_f)
\Lambda^{e_i}| = \infty$, so there exists $\tau_i \in s(\tau_f)
\Lambda^{e_i}$ such that $\psi \big( d(\alpha \tau_f), d(\alpha
\tau_f\tau_i) \big) \neq \tau_i$ for all $\psi \in \MCE(\alpha
\tau_f, \beta \tau_f)$.

Let $\tau = \tau_f \tau_i$. We now must argue that $\MCE(\alpha
\tau, \beta \tau) = \emptyset$. Suppose for contradiction that
$\psi \in \MCE(\alpha \tau, \beta \tau)$. Then in particular
$\psi \big( 0, d(\alpha \tau) \vee d(\beta \tau) \big) \in
\MCE(\alpha \tau_f, \beta \tau_f)$. However $\psi \big(
d(\alpha \tau_f), d(\alpha \tau_f \tau_i) \big) = \tau_i$ which
contradicts our choice of $\tau_i$.

\textbf{Case (1)(ii).} Let $\lambda$ be as in Case (1)(ii). Fix
$\lambda' \in s(\lambda) \Lambda^{\leq D - p} \subseteq
s(\lambda) \Lambda$. Then $d(\lambda \lambda') \ngeq D$ and
there exists $i \in \{ 1,\ldots,k \}$ such that $d(\lambda
\lambda')_i < D_i$.  For each $i$ with $d(\lambda \lambda')_i <
D_i$, we have $d(\lambda')_i < (D - p)_i$ so by definition of
$\Lambda^{\leq D - p}$ we have $s(\lambda') \Lambda^{e_i} =
\emptyset$. We now claim that there exists $i$ such that
$d(\lambda \lambda')_i < D_i$ and $d(\alpha)_i \neq
d(\beta)_i$. Suppose for contradiction that $d(\alpha)_i =
d(\beta)_i$ for each $i \in \{1,\ldots,k\}$ such that
$d(\lambda \lambda') < D_i$. By assumption we have $d(\alpha)
\wedge d(\beta) = 0$ which implies that whenever $d(\alpha)_i =
d(\beta)_i$ we must have $d(\alpha)_i = d(\beta)_i = 0$. Fix $i
\in \{ 1,\ldots,k \}$ such that $d(\lambda \lambda')_i < D_i$.
Then $d(\alpha)_i = d(\beta)_i = 0$ and we have the
contradiction
\[
    0 < d(\lambda \lambda')_i < D_i = \big( d(\alpha) \vee d(\beta) \big)_i = 0.
\]
We now let $\tau = \lambda \lambda'$. Then $d(\alpha \tau)_i
\neq d(\beta \tau)_i$ but $s(\tau) \Lambda^{e_i} = s(\lambda')
\Lambda^{e_i} = \emptyset$ and hence $\MCE(\alpha \tau, \beta
\tau) = \emptyset$ as required.
\end{proof}

\section{Consequences of Aperiodicity}\label{is}

We now characterise aperiodicity of $\Lambda$ in terms of the
ideal structure of $C^*(\Lambda)$ and prove a version of the
Cuntz-Krieger uniqueness theorem.

\begin{theorem}\label{4.3rob-s2}
Let $(\Lambda, d)$ be a finitely aligned $k$-graph. Then the
following are equivalent:
\begin{enumerate}
\item $\Lambda$ is aperiodic.
\item Every non-zero ideal of $C^*(\Lambda)$ contains a
    vertex projection.
\item The $\partial \Lambda$ representation $\pi_S$ is
    faithful.
\end{enumerate}
\end{theorem}

The bulk of the work goes into (1) $\Rightarrow$ (2). This is
the Cuntz-Krieger uniqueness theorem and we prove it in the
next subsection. The implication $(2) \Rightarrow (3)$ follows
from Remark~\ref{0rep}. We therefore begin by proving $(3)
\Rightarrow (1)$. We first establish two preliminary results.

\begin{lemma}\label{*}
Let $\Lambda$ be a finitely aligned $k$-graph and fix $v \in
\Lambda^0$. Suppose that $\Lambda$ has local periodicity $m,n$
at $v$. If $x \in v (\partial \Lambda)$ and $i \in
\{1,\ldots,k\}$ satisfy $d(x)_i < \infty$, then $m_i = n_i$.
\end{lemma}
\begin{proof}
Suppose $x \in v (\partial \Lambda)$ and $i \in \{1,\ldots,k\}$
satisfy $d(x)_i < \infty$. Since $\sigma^{m} (x) = \sigma^{n}
(x)$ we have $d(x) - m = d(x) - n$ and in particular $\big(
d(x) - m \big)_i = \big( d(x) - n \big)_i$ which implies $m_i =
n_i$ since $d(x)_i < \infty$.
\end{proof}

\begin{lemma}\label{3.3}
Let $\Lambda$ be a finitely aligned $k$-graph. Suppose
$\Lambda$ has local periodicity $m,n$ at $v$. Then there exist
$\mu, \nu, \alpha \in \Lambda$ such that $r(\mu) = r(\nu) = v$,
$s(\mu) = s(\nu) = r(\alpha)$, $d(\mu) = m$, $d(\nu) = n$ and
$\mu \alpha z = \nu \alpha z$ for all $z \in s(\alpha) \partial
\Lambda$.
\end{lemma}
\begin{proof}
Fix $x \in v (\partial \Lambda)$. Since $\Lambda$ has local
periodicity $m,n$ at $v$ we know that $m \vee n \leq d(x)$. Let
$\mu = x(0, m)$, $\nu = x(0, n)$ and $\alpha = x(m, m \vee n)$.
Then $s(\mu) = r \big( \sigma^m (x) \big) = r \big( \sigma^n
(x) \big) = s(\nu)$. Fix $z \in s(\alpha) \partial \Lambda$.
Since $(\mu \alpha z)(0, n) = \nu$ and $d(\mu) = m$ we have
\begin{flalign*}
&&&\mu \alpha z = \nu \sigma^n (\mu \alpha z) = \nu \sigma^m
(\mu \alpha z) = \nu \alpha z. &\qedhere
\end{flalign*}
\end{proof}

Recall that given a finitely aligned $k$-graph $\Lambda$, there
is a strongly continuous \emph{gauge action} $\gamma : \T
\rightarrow \Aut \big( C^*(\Lambda) \big)$ given using
multi-index notation by $\gamma_z(s_\lambda) = z ^{d(\lambda)}
s_\lambda$.

\begin{proof}[Proof of $(3) \Rightarrow (1)$ in Theorem~\ref{4.3rob-s2}]\label{pfof3=>1}
We prove the contrapositive statement. Suppose $\Lambda$ is not
aperiodic. Then Proposition~\ref{teqnlp} implies that there
exist $v \in \Lambda^0$ and distinct $m, n \in \N^k$ such that
$\Lambda$ has local periodicity $m, n$ at $v$. By
Lemma~\ref{3.3} there exist $\mu, \nu, \alpha \in \Lambda$ such
that $r(\mu) = r(\nu) = v$, $s(\mu) = s(\nu) = r(\alpha)$,
$d(\mu) = m$, $d(\nu) = n$ and $\mu \alpha y = \nu \alpha y$
for every $y \in s(\alpha) \partial \Lambda$.

We claim that $a := s_{\mu \alpha} s_{\mu \alpha}^* - s_{\nu
\alpha} s_{\mu \alpha}^* \in \ker(\pi_S) \backslash \{0\}$. To
see $a \neq 0$ we check that for $\omega \in \mathbb{T}^k$ with
$\omega^{d(\nu) - d(\mu)} = -1$ we have $(\operatorname{id} +
\gamma_\omega) (a) = 2 s_{\mu \alpha} s_{\mu \alpha}^* \neq 0$.
To see that $\pi_S (a) = 0$ we check directly using our choice
of $\mu, \nu, \alpha$ that $\pi_S (a) \zeta_x = 0$ for all $x
\in \partial \Lambda$. The details are the same as
\cite[Proposition~3.5]{robsi1}.
\end{proof}

\subsection{The Cuntz-Krieger Uniqueness Theorem}

We now use our definition of aperiodicity to prove a version of
the Cuntz-Krieger uniqueness theorem. We start with a technical
lemma.

\begin{lemma}\label{pre_ck}
Suppose $(\Lambda, d)$ is an aperiodic finitely aligned
$k$-graph. Fix $v \in \Lambda^0$ and let $H$ be a finite subset
of $\Lambda v$. Then there exists $\tau \in v \Lambda$ such
that $\MCE(\alpha \tau, \beta \tau) = \emptyset$ for every pair
of distinct paths $\alpha, \beta \in H$.
\end{lemma}
\begin{proof}
We proceed by induction on $|H|$. If $|H| = 2$ then
Lemma~\ref{pre_ck} reduces to the definition of aperiodicity.

Now suppose the result is true whenever $|H| = k$. Fix $H
\subseteq \Lambda v$ with $|H| = k + 1$. Fix $\alpha \in H$ and
list $H \backslash \{ \alpha \} = \{ \beta_1,\ldots, \beta_k
\}$. By the inductive hypothesis there exists $\tau_\alpha \in
v \Lambda$ such that $\MCE(\beta_1 \tau_\alpha, \beta_2
\tau_\alpha) = \emptyset$ for every distinct $\beta_1, \beta_2
\in H \backslash \{ \alpha \}$. Inductively applying
aperiodicity, we obtain paths $\tau_1,\dots, \tau_k$ such that
for each $j \leq k$,
\begin{align*}
\MCE(\alpha \tau_\alpha \tau_1 \ldots \tau_j,\; \beta_j \tau_\alpha \tau_1 \ldots \tau_j) = \emptyset.
\end{align*}

We claim that $\tau = \tau_\alpha \tau_1 \ldots \tau_k$
satisfies $\MCE(\mu \tau, \nu \tau) = \emptyset$ for all
distinct $\mu, \nu \in H$. Fix distinct $\mu, \nu \in H$. First
suppose that $\mu, \nu \in H \setminus \{ \alpha \}$. Then
$\MCE(\mu \tau_\alpha, \nu \tau_\alpha) = \emptyset$ and since
$\tau_\alpha$ is an initial segment of $\tau$ it follows that
$\MCE(\mu \tau, \nu \tau) = \emptyset$. Now suppose that one of
$\mu, \nu$ is equal to $\alpha$; without loss of generality
suppose $\mu = \alpha$. It remains to show that $\MCE(\alpha
\tau, \nu \tau) = \emptyset$. Since $\nu \neq \alpha$ we have
$\nu = \beta_p$ for some $1 \leq p \leq k$. By construction
$\MCE(\alpha \tau_\alpha \tau_1 \ldots \tau_p, ~ \beta_p
\tau_\alpha \tau_1 \ldots \tau_p) = \emptyset$. Since $\tau_1
\ldots \tau_p$ is an initial segment of $\tau$ it follows that
$\MCE(\alpha \tau, \nu \tau) = \emptyset$.
\end{proof}

The following technical lemma allows us to replace the use of
condition~(B) in the proof of the Cuntz-Krieger uniqueness
theorem \cite[Theorem~4.5]{rsy2} with our aperiodicity
hypothesis.

\begin{lemma}\label{C}
Let $\{ t_\lambda : \lambda \in \Lambda \}$ be a Cuntz-Krieger
$\Lambda$-family with $t_v \neq 0$ for each $v \in \Lambda^0$.
Fix $v \in \Lambda^0$ and let $H$ be a finite subset of
$\Lambda v$. Fix $N \in \N^k$ and a linear combination $a =
\sum_{\mu, \nu \in H} a_{\mu, \nu} t_\mu t_\nu^*$ such that
whenever $a_{\mu, \nu} \neq 0$ we have $\mu \in \Lambda^{\leq
N}$. Also let $a_0 = \sum_{\mu, \nu \in H, d(\mu) = d(\nu)}
a_{\mu, \nu} t_\mu t_\nu^*$. Then there exists a
norm-decreasing linear map $Q : C^*(\Lambda) \rightarrow
C^*(\Lambda)$ such that $\| Q(a_0) \| = \| a_0 \|$ and $Q(a_0)
= Q(a)$. In particular, $\| a_0 \| = \| a \|$.
\end{lemma}

\begin{remark*}
Lemma 4.11 of \cite{rsy2} claims that $Q$ maps $\pi \big(
C^*(\Lambda) \big)$ to $\pi \big( C^*(\Lambda)^\gamma \big)$.
However this assertion is not proved and is not obviously true.
Fortunately it is also not needed.
\end{remark*}

For the proof of Lemma~\ref{C} the following remark will prove
useful.

\begin{remark}\label{B}
If $\alpha, \beta \in \Lambda$ satisfy $s(\alpha) = s(\beta)$,
$d(\alpha) = d(\beta)$ and $\alpha \in \Lambda^{\leq n}$ then
$\beta \in \Lambda^{\leq n}$ also. To see this suppose
$d(\beta)_i < n_i$. Then $d(\alpha)_i = d(\beta)_i < n_i$, and
$\alpha \in \Lambda^{\leq n}$ forces $s(\beta)\Lambda^{e_i} =
s(\alpha)\Lambda^{e_i} = \emptyset$.
\end{remark}

\begin{proof}[Proof of Lemma~\ref{C}]
Since $H \subseteq \Lambda v$ is finite, Lemma~\ref{pre_ck}
implies that there exists $\tau \in v\Lambda$ such that
$\MCE(\alpha \tau, \beta \tau) = \emptyset$ for every pair of
distinct paths $\alpha, \beta \in H$. For each $n \leq N$ let
\begin{align*}
    Q_n &:= \sum_{\rho \in H \cap \Lambda^{\leq N} \cap \Lambda^n} t_{\rho \tau} t_{\rho \tau}^*
\intertext{and define $Q : C^*(\Lambda) \rightarrow C^*(\Lambda)$ by}
    Q(b) &:= \sum_{n \leq N} Q_n b Q_n.
\end{align*}
Lemma~\ref{2.7} implies that the $Q_n$ are mutually orthogonal
projections, so $Q \big( C^*(\Lambda) \big) \cong \oplus_{n
\leq N} Q_n C^*(\Lambda) Q_n$. Hence $\| Q(a) \| = \max_{n \leq
N} \| Q_n a Q_n \| \leq \|a\|$, so $Q$ is norm-decreasing; it
is clearly linear.

We will now show that $\| Q(a_0) \| = \| a_0 \|$. Consider
\[
G = \lsp \{ t_\alpha t_\beta^* : \alpha, \beta \in H, d(\alpha) = d(\beta), \alpha \in \Lambda^{\leq N} \}.
\]
We will show that $G$ is a finite-dimensional $C^*$-subalgebra
of $C^*(\Lambda)$ and use this to see that $Q$ restricts to an
isomorphism of $G$ onto $Q(G)$. It then follows that $\| Q(a_0)
\| = \| a_0 \|$.

For each $n \leq N$ let $F_n$ denote the set $\{ t_\alpha
t_\beta^* : \alpha, \beta \in H, d(\alpha) = n = d(\beta),
\alpha \in \Lambda^{\leq N} \}$. Fix $n \leq N$ and elements
$t_\alpha t_\beta^*$ and $t_\mu t_\nu^*$ of $F_n$. Then $\beta,
\mu \in \Lambda^n$ and (CK3) forces
\begin{align*}
(t_\alpha t_\beta^*)(t_\mu t_\nu^*) = t_\alpha (\delta_{\beta, \mu} t_v) t_\nu^* = \delta_{\beta, \mu} t_\alpha t_\nu^*.
\end{align*}
A standard argument using (CK3) and that $t_v \neq 0$ shows
that each $t_\alpha t_\beta^*$ is nonzero. Hence $t_\alpha
t_\beta^* \mapsto \theta_{\alpha, \beta}$ determines an
isomorphism $G_n := \lsp(F_n) \cong M_{H \cap \Lambda^{\leq N}
\cap \Lambda^n} (\C)$.

Fix distinct $m, n \leq N$. Let $t_\alpha t_\beta^* \in F_m$
and $t_\mu t_\nu^* \in F_n$. Since $\alpha \in \Lambda^{\leq
N}$ and $d(\alpha) = d(\beta)$, Remark~\ref{B} implies $\beta
\in \Lambda^{\leq N}$. Since $d(\beta) \neq d(\mu)$, we have
$\beta \neq \mu$, so Lemma~\ref{2.7} (ii) implies that
$(t_\alpha t_\beta^*)(t_\mu t_\nu^*) = 0$. Hence $G_m \perp
G_n$. Thus $G = \oplus_{n \leq N} G_n \cong \oplus_{n \leq N}
M_{H \cap \Lambda^{\leq N} \cap \Lambda^n} (\C)$ via $t_\alpha
t_\beta^* \mapsto \theta_{\alpha, \beta}$.

For $\alpha, \beta \in H$ such that $a_{\alpha, \beta} \neq
\emptyset$, we have $\alpha \in \Lambda^{\leq N}$ and we
calculate
\begin{align}
Q(t_{\alpha} t_{\beta}^*)
    &= \sum_{n \leq N} Q_n (t_{\alpha} t_{\beta}^*) Q_n \nonumber \\
    &= \sum_{n \leq N}
        \Big( \sum_{\rho \in H \cap \Lambda^{\leq N} \cap \Lambda^n} t_{\rho \tau} t_{\rho \tau}^* \Big)
        (t_{\alpha} t_{\beta }^*)
        \Big( \sum_{\rho' \in H \cap \Lambda^{\leq N} \cap \Lambda^n} t_{\rho' \tau} t_{\rho' \tau}^* \Big) \nonumber \\
    &= \sum_{n \leq N}
        \Big( \sum_{\rho, \rho' \in H \cap \Lambda^{\leq N} \cap \Lambda^n}
            t_{\rho \tau} t_\tau^* (t_\rho^* t_{\alpha}) (t_{\beta }^* t_{\rho'}) t_\tau t_{\rho' \tau}^*
        \Big) \nonumber \\
    &= \sum_{ \rho' \in H \cap \Lambda^{\leq N} \cap \Lambda^{d(\alpha)}}
            t_{\alpha \tau} (t_{\beta \tau}^* t_{\rho' \tau}) t_{\rho' \tau}^*
            \quad\text{by Lemma~\ref{2.7}(ii)}. \label{1}
\end{align}
If, in addition, $d(\alpha) = d(\beta)$, then $\beta \in
\Lambda^{\leq N}$ by Remark~\ref{B}, so Lemma~\ref{2.7}(ii)
implies $t_\beta^* t_{\rho'} = 0$ unless $\rho' = \beta$. So
continuing our calculation from above, with  $\alpha, \beta \in
H$, $\alpha \in \Lambda^{\leq N}$ and $d(\alpha) = d(\beta)$,
we have
\[
Q(t_{\alpha} t_{\beta}^*)
    = t_{\alpha \tau} t_{\alpha \tau}^* t_\alpha t_\beta^* t_{\beta \tau} t_{\beta \tau}^*
    = t_{\alpha \tau} t_{\beta \tau}^*
\]
In particular, for $t_\alpha t_\beta^*, t_\mu t_\nu^* \in
\bigcup_{n \leq N} F_n$,
\[
Q(t_{\alpha} t_{\beta}^*) Q(t_\mu t_\nu^*)
    = t_{\alpha \tau} t_{\beta \tau}^* t_{\mu \tau} t_{\nu \tau}^*
    = t_{\alpha \tau} t_\tau^* (t_\beta^* t_\mu) t_\tau t_{\nu \tau}^*
    = \delta_{\beta, \mu} t_{\alpha \tau} t_{\nu \tau}^*.
\]
So the set $\{ Q(t_{\alpha} t_{\beta}^*) : \alpha, \beta \in H,
d(\alpha) = d(\beta), \alpha \in \Lambda^{\leq N} \}$ is a
system of non-zero matrix units for an isomorphic copy of
$\oplus_{n \leq N} M_{\{ \alpha \tau : \alpha \in H \cap
\Lambda^{\leq N} \cap \Lambda^n \}} (\C)$. In particular
$t_\alpha t_\beta^* \mapsto t_{\alpha \tau} t_{\beta \tau}^*$
determines an isomorphism of $F_n$. Thus $Q$ restricts to an
isomorphism of $G$ onto $Q(G)$ forcing $\|Q(a_0)\| = \|a_0\|$.

It remains to show that $Q(a_0) = Q(a)$. Fix $\alpha, \beta \in
H$ such that $d(\alpha) \neq d(\beta)$ and $a_{\alpha, \beta}
\neq 0$. It suffices to show that $Q(t_\alpha t_\beta^*) = 0$.
Using~\eqref{1} we have
\[
Q(t_{\alpha} t_{\beta}^*)
    = \sum_{\rho' \in H \cap \Lambda^{\leq N} \cap \Lambda^n} t_{\alpha \tau} t_{\beta \tau}^* t_{\rho' \tau} t_{\rho' \tau}^*.
\]
Since $d(\beta) \neq d(\alpha)$ we have $\beta \neq \rho'$ for
all $\rho' \in H \cap \Lambda^{\leq N} \cap
\Lambda^{d(\alpha)}$. Thus $\MCE(\beta \tau, \rho' \tau) =
\emptyset$ for every $\rho' \in H \cap \Lambda^{\leq N} \cap
\Lambda^{d(\alpha)}$ by choice of $\tau$. Thus (CK3) forces
$Q(t_{\alpha} t_{\beta}^*) = 0$.
\end{proof}

We can now prove our version of the Cuntz-Krieger Uniqueness
Theorem.

\begin{theorem}[The Cuntz-Krieger Uniqueness Theorem]\label{ck}
Suppose $(\Lambda, d)$ is an aperiodic, finitely aligned
$k$-graph and suppose $\pi$ is a representation of
$C^*(\Lambda)$ such that $\pi(s_v) \neq 0$ for every $v \in
\Lambda^0$. Then $\pi$ is faithful.
\end{theorem}
\begin{proof}
The opening paragraph of \cite[section~4]{rsy2} together with
\cite[Proposition~4.1]{rsy2} show that it suffices to fix a
finite subset $H \subseteq \Lambda$ and scalars $\{ a_{\lambda,
\mu} : \lambda, \mu \in H \}$ and show that
\[
\Big\| \sum_{\lambda, \mu \in H, d(\lambda) = d(\mu)} a_{\lambda, \mu} \pi(s_\lambda s_\mu^*) \Big\|
    \leq \Big\| \sum_{\lambda, \mu \in H} a_{\lambda, \mu} \pi(s_\lambda s_\mu^*) \Big\|
\]
Equation~(4.4) and Proposition~4.10 of \cite{rsy2} show that we
may assume that there exist $v \in \Lambda^0$ and $N \in \N^k$
such that $H \subseteq v \Lambda$ and also that $a_{\lambda,
\mu} \neq \emptyset$ implies $\lambda \in \Lambda^{\leq N}$. We
are now in the situation of Lemma~\ref{C} with $t_\lambda =
\pi(s_\lambda)$ for every $\lambda \in \Lambda$, $a =
\sum_{\lambda, \mu \in H} a_{\lambda, \mu} t_\lambda t_\mu^*$
and $a_0 = \sum_{\lambda, \mu \in H, d(\lambda)= d(\mu)}
a_{\lambda, \mu} t_\lambda t_\mu^*$. Thus Lemma~\ref{C} gives
\[
\|a_0\| = \|Q(a_0)\| = \|Q(a)\| \leq \|a\|.
\]
as required.
\end{proof}

\begin{remark*}
We have really just recycled the proof of
\cite[Theorem~4.5]{rsy2}, which occupies all of
\cite[Section~4.2]{rsy2}, replacing \cite[Lemma~4.11]{rsy2} (in
which condition~(B) was invoked) with our Lemma~\ref{C} which
uses aperiodicity instead.
\end{remark*}

\begin{proof}[Proof of Theorem~\ref{4.3rob-s2}]
(1) $\Rightarrow$ (2) follows from Theorem~\ref{ck}, we proved
(3) $\Rightarrow$ (1) on page~\pageref{pfof3=>1} and (2)
$\Rightarrow$ (3) follows from Remark~\ref{0rep}.
\end{proof}

\section{Consequences of cofinality}\label{prf}

In this section we characterise cofinality in $\Lambda$ in
terms of the ideal structure in $C^*(\Lambda)$, and conclude by
proving our main result, Theorem~\ref{main}. The meat of
Theorem~\ref{ideal}, namely $(2) \Rightarrow (3)$, is identical
to \cite[Proposition~4.3]{shotwell} and we have only included
the argument for completeness. The key points of difference are
our formulation of cofinality in terms of finite paths, and the
fact that our proof is direct: our arguments do not appeal to
the classification of gauge-invariant ideals in $C^*(\Lambda)$
of \cite{sims2}.

\begin{theorem}\label{ideal}
Suppose $\Lambda$ is a finitely aligned $k$-graph. Then the
following are equivalent.
\begin{enumerate}
\item $\Lambda$ is cofinal;
\item for each $v \in \Lambda^0$ and each $x \in \partial
    \Lambda$ there exists $n \leq d(x)$ such that $v
    \Lambda x(n) \neq \emptyset$;
\item the only ideal of $C^*(\Lambda)$ which contains $s_v$
    for some $v \in \Lambda^0$, is $C^*(\Lambda)$; and
\item the only ideal of $C^*(\Lambda)$ which nontrivially
    intersects $C^*(\Lambda)^\gamma$ is $C^*(\Lambda)$.
\end{enumerate}
\end{theorem}

We begin with a technical lemma. Conditions (1) and (2) of the
following Lemma are precisely conditions (MT1) and (MT2) of
\cite[Proposition~5.5.3]{ASPhD}.

\begin{lemma}\label{E}
Let $K \subseteq \Lambda^0$ be a nonempty set such that
\begin{enumerate}
\item If $u \in K$ and $E \in u \FE(\Lambda)$ then there
    exists $\alpha \in E$ such that $s(\alpha) \in K$; and
\item If $u \in K$ and $v \Lambda u \neq \emptyset$ then $v
    \in K$.
\end{enumerate}
Then there exists $x \in \partial \Lambda$ such that $x(n) \in
K$ for every $n \leq d(x)$.
\end{lemma}
\begin{proof}

We draw heavily on the techniques used in the proof of
\cite[Lemma~4.7]{sims1}. Define $P : (\N \setminus \{ 0 \})^2
\rightarrow \N \setminus \{ 0 \}$ by
\[
P(m, n) := \frac{(m + n - 1)(m + n - 2)}{2} + m.
\]
Then $P$ is the position function for the diagonal listing
\[
\{ (1, 1), (1, 2), (2, 1), (1, 3), (2, 2), (3, 1), (1, 4)... \} ~ \mathrm{of} ~ (\N \setminus \{ 0 \})^2.
\]
That is, if $P(m, n) = l$ then $(m, n)$ is the $l$th term in
the above listing. For each $l \in \N \setminus \{ 0 \}$ let
$(i_l, j_l) \in (\N \setminus \{ 0 \})^2$ be the unique pair
such that $P(i_l, j_l) = l$.

Fix $v \in K$. We claim there exists a sequence $(\lambda_l)_{l
= 1}^\infty \subset v \Lambda$ and for each $l$ an enumeration
$s(\lambda_l)\FE(\Lambda) = \{ E_{l, j} : j \geq 1 \}$ which
satisfy
\begin{enumerate}
\item[(i)] $\lambda_{l + 1} \big( 0, d(\lambda_l) \big) =
    \lambda_l$ for every $l \geq 1$ and,
\item[(ii)] $\lambda_{l + 1} \big( d(\lambda_{i_l}),
    d(\lambda_{l + 1}) \big) \in E_{i_l, j_l} \Lambda K$
    for every $l \geq 1$.
\end{enumerate}

We proceed by induction on $l$. For $l = 0$ define $\lambda_{l
+ 1} = \lambda_1 = v$. Since $\Lambda$ is countable, for each
$w \in \Lambda^0$ the collection of finite subsets of $w
\Lambda$ is countable. In particular, $w \FE(\Lambda)$ is
countable. Let $\{ E_{1, j} : j \in \N \setminus \{ 0 \} \}$ be
a listing of $v \FE(\Lambda)$. Then (i) and (ii) are trivially
satisfied because $l = 0 < 1$.

Now suppose $l \geq 1$ and that $\lambda_n$ and $\{ E_{n, j} :
j \geq 1 \}$ satisfy (i) and (ii) for $1 \leq n \leq l$. Recall
that $P (i_l, j_l) = l$ so $i_l$ is the horizontal co-ordinate
of the $l$th term in the diagonal listing. In particular $i_l <
l$ and by assumption the listing $\{ E_{i_l, j} : j \geq 1 \}$
of $s(\lambda{i_l}) \FE(\Lambda)$ satisfies (i) and (ii). In
particular the set $E_{i_l, j_l}$ has already been fixed. We
now must find $\lambda_{l + 1}$ satisfying (i) and (ii).

Let $\mu := \lambda_l \big( d(\lambda{i_l}), d(\lambda_l)
\big)$ and $E := \Ext(\mu, E_{i_l, j_l})$. Lemma~\ref{c5}
implies that $E \in s(\mu)\FE(\Lambda)$. Condition (ii) implies
that $s(\mu) \in K$, so hypothesis (1) implies that there
exists $\alpha \in E$ such that $s(\alpha) \in K$. By
definition of $E$ there exist $\nu \in E_{i_l, j_l}$ and $\beta
\in \Lambda$ such that $\mu \alpha = \nu \beta \in \MCE(\mu,
\nu)$. Then $\lambda_{l + 1} := \lambda_l \alpha$ satisfies
(ii) and trivially satisfies (i). This proves the claim.

Define $m \in (\N \cup \{ \infty \})^k$ by $m_i = \sup\{
d(\lambda_l)_i : l \geq 1 \}$, and define $x : \Omega_{k, m}
\rightarrow \Lambda$ by $x \big( 0, d(\lambda_l) \big) =
\lambda_l$ for all $l$. Once we show that $x \in v(\partial
\Lambda)$, hypothesis (2) will force $x(n) \in K$ for each $n
\leq d(x)$ as required.

Fix $n \leq d(x)$ and $F \in x(n) \FE(\Lambda)$. Fix $l$ such
that $d(\lambda_{i_l}) \geq n$ and let $E := \Ext \big( x(n,
d(\lambda_{i_l})), F \big)$. Then Lemma~\ref{c5} implies that
$E \in s(\lambda_{i_l}) \FE(\Lambda)$. By definition of the
position function $P$, there exists $k \geq l$ such that $i_k =
l$ and hence $E = E_{l, j_k}$. Then condition (ii) implies
\[
x \big( d(\lambda_l, d(\lambda_{k + 1}) \big) = x \big( d(\lambda_{i_k}, d(\lambda_{k + 1}) \big) \in E.
\]
On the other hand, since $E = \Ext \big( x(n,
d(\lambda_{i_l})), F \big)$ the definition of $\Ext$ on
page~\pageref{ext} implies that there exists $\mu \in F$ and
$\alpha \in \Lambda$ such that
\[
x \big( n, d(\lambda_l) \big) x \big( d(\lambda_l), d(\lambda_{k + 1}) \big) = \mu \alpha \in \MCE \big( x(n, d(\lambda_l)), \mu \big)
\]
In particular
\[
x \big( n, n + d(\mu) \big) = \Big( x \big( n, d(\lambda_{k + 1}) \big) \Big) \big( 0, d(\mu) \big) = (\mu \alpha) \big( 0, d(\mu) \big) = \mu.
\]
Hence $x \in \partial \Lambda$ as required.
\end{proof}

\begin{proof}[Proof of Theorem~\ref{ideal}]
$(1) \Rightarrow (2)$: Suppose $\Lambda$ is cofinal. Fix $v \in
\Lambda^0$ and $x \in \partial \Lambda$. Since $\Lambda$ is
cofinal there exists $E \in x(0) \FE(\Lambda)$ such that $v
\Lambda s(\alpha) \neq \emptyset$ for all $\alpha \in E$. Since
$x \in \partial \Lambda$ there exists $\alpha \in E$ such that
$x \big( 0, d(\alpha) \big) = \alpha$. In particular $n =
d(\alpha)$ satisfies
\[
x(n)  = s \big( x \big( 0, d(\alpha) \big) \big) = s(\alpha)
\]
so $v \Lambda x(n) \neq \emptyset$.

$(2) \Rightarrow (1)$: We prove the contrapositive statement.
Suppose there exist $v, w \in \Lambda^0$ such that for every $E
\in w \FE(\Lambda)$ there exists $\alpha \in E$ such that $v
\Lambda s(\alpha) = \emptyset$. We aim to apply Lemma~\ref{E}.
To this end we let
\[
K := \{ u \in \Lambda^0 : \text{ for each }E \in u\FE(\Lambda) \text{ there exists }\alpha \in E \text{ such that } v \Lambda s(\alpha) = \emptyset \}.
\]
We claim that $K$ satisfies the hypothesis of Lemma~\ref{E}. We
have $K \neq \emptyset$ since $w \in K$. Fix $u \in K$ and $E
\in u\FE(\Lambda)$. To show that $K$ satisfies hypothesis (1)
of Lemma~\ref{E} we claim there exists $\alpha \in E$ such that
$s(\alpha) \in K$. Indeed, if for every $\alpha \in E$ we have
$s(\alpha) \notin K$ then for every $\alpha \in E$ there exists
$F_\alpha \in s(\alpha)\FE(\Lambda)$ such that $v \Lambda
s(\eta) \neq \emptyset$ for every $\eta \in F_\alpha$. Let $G
:= \{ \alpha \eta : \alpha \in E, ~ \eta \in F_\alpha \}$ and
$F := \{ \alpha \in E : F_\alpha ~ \mathrm{does ~ not ~
contain} ~ s(\alpha) \}$. Then since $F_\alpha \in
\FE(\Lambda)$ for all $\alpha \in E$
\cite[Definition~5.2]{sims1}, combined with
\cite[Lemma~5.3]{sims1} implies $G \in \FE(\Lambda)$. By
construction of $G$ we have $v \Lambda s(\lambda) \neq
\emptyset$ for each $\lambda \in G$ and since $G \in
u\FE(\Lambda)$ this contradicts $u \in K$. Hence $K$ satisfies
hypothesis (1) of Lemma~\ref{E}.

For hypothesis (2) fix $u \in K$ and suppose $v \Lambda u \neq
\emptyset$, say $\lambda \in v \Lambda u$. We must show that $v
\in K$. Fix $E \in v\FE(\Lambda)$ and let $F = \Ext(\lambda,
E)$. It follows from Lemma~\ref{c5} that $F \in u\FE(\Lambda)$.
Since $u \in K$ there exists $\alpha \in F$ such that $v
\Lambda s(\alpha) = \emptyset$. By definition of $F$ there
exists $\mu \in E$ and $\nu \in \MCE(\lambda, \mu)$ such that
$\alpha = \nu \big( d(\lambda), d(\lambda) \vee d(\mu) \big)$.
In particular
\[
\beta := \nu \big( d(\mu), d(\lambda) \vee d(\mu) \big)
    \in s(\mu) \Lambda s(\alpha)
\]
and since $v \Lambda s(\alpha) = \emptyset$ we have $v \Lambda
s(\mu) = \emptyset$. That is, $\mu \in E$ satisfies $v \Lambda
s(\mu) = \emptyset$. Hence $v \in K$ and thus $K$ satisfies the
hypotheses of Lemma~\ref{E} as claimed. By Lemma~\ref{E} there
exists $x \in \partial \Lambda$ with $x(n) \in K$ for every $n
\leq d(x)$. Since $\{ x(n) \} \in x(n) \FE(\Lambda)$ for all $n
\leq d(x)$, the definition of $K$ then implies that $v \Lambda
x(n) = \emptyset$ for all $n$.

$(2) \Rightarrow (3)$: Fix an ideal $I$ of $C^*(\Lambda)$ with
$s_w \in I$ for some $w \in \Lambda^0$. Let $H = \{ w : s_w \in
I \}$. We will show $\Lambda^0 \setminus H$ satisfies the
hypothesis of Lemma~\ref{E} and then deduce that $H =
\Lambda^0$.

Recall from \cite[Definition~3.3]{rsy2} that $\Pi E$ was
defined as the smallest set containing $E$ such that
\[
\parbox{0.8\textwidth}{for all $\lambda,\mu,\nu,\rho \in
    \Pi E$ with $d(\lambda) = d(\mu)$, $d(\nu) = d(\rho)$,
    $s(\lambda) = s(\mu)$ and $s(\nu) = s(\rho)$ and each
    $\mu\alpha = \nu\beta \in \MCE(\mu,\nu)$, we have
    $\lambda\alpha, \rho\beta \in \Pi E$.}
\]
In particular, the construction of the set $F$ in the proof of
\cite[Lemma~3.2]{rsy2} combines with
\cite[Definition~3.3]{rsy2} to show that each element $\mu$ of
$\Pi E$ has the form $\mu = \nu \nu'$ for some $\nu \in E$.

We claim that if $v \in \Lambda^0 \setminus H$ then for every
finite exhaustive set $F \subset v \Lambda$ there exists
$\lambda \in F$ such that $s(\lambda) \in \Lambda^0 \setminus
H$. Suppose for contradiction that $s(\nu) \in H$ for each $\nu
\in F$. Fix $\mu \in \Pi F$ then $\mu = \nu \nu'$ for some $\nu
\in F$. Since $s_{s(\nu)} \in I$ we have $s_{s(\mu)} =
s_{\nu'}^* s_{s(\nu)} s_{\nu'} \in I$ for every $\mu \in \Pi
F$. Then $s_\mu s_\mu^* = s_\mu s_{s(\mu)} s_\mu^* \in I$.
Proposition~3.5 of \cite{rsy2} implies that
\[
s_v = \prod_{\mu' \in v \Pi F}(s_v - s_{\mu'} s_{\mu'}^*) +
    \sum_{\mu \in v \Pi F} \big(
        s_\mu s_\mu^* \prod_{\nu \nu' \in \Pi F, d(\nu') > 0}
            (s_\nu s_\nu^* - s_{\nu \nu'} s_{\nu \nu'}^*)
    \big).
\]
Since $F$ is exhaustive, $\Pi F$ is also exhaustive and
\cite[Lemma~3.2]{rsy2} implies that $\Pi F$ is finite.
Hence~(CK4) implies that $\prod_{\mu' \in v \Pi F}(s_v -
s_{\mu'} s_{\mu'}^*) = 0$ and thus
\[
s_v = \sum_{\mu \in v \Pi F}
    \big(
    s_\mu s_\mu^* \prod_{\nu \nu' \in \Pi F, d(\nu') > 0}
        (s_\nu s_\nu^* - s_{\nu \nu'} s_{\nu \nu'}^*)
    \big)
    \in I.
\]
This contradicts $v \in \Lambda^0 \setminus H$. Thus $\Lambda^0
\setminus H$ satisfies hypothesis (1) of Lemma~\ref{E}. For
hypothesis (2), suppose $u \in \Lambda^0 \setminus H$ and
$\lambda \in v \Lambda u$. Suppose for contradiction that $v
\in H$. Then $s_u = s_{s(\lambda)} = s_\lambda^* s_v s_\lambda
\in H$ contradicting the definition of $u$. So $\Lambda^0
\setminus H$ satisfies the hypothesis of Lemma~\ref{E}.

By Lemma~\ref{E} there exists $x \in \partial \Lambda$ with
$x(n) \in \Lambda^0 \setminus H$ for every $n \leq d(x)$. By
hypothesis there exists $n \leq d(x)$ such that $w \Lambda x(n)
\neq \emptyset$. Let $\mu \in w \Lambda x(n)$. Since $s_w \in
I$ we have $s_{x(n)} = s_\mu^* s_w s_\mu \in I$ forcing $x(n)
\in H$. This contradicts the definition of $x$. Hence $H =
\Lambda^0$. For each $\mu \in \Lambda$ we now have $s(\mu) \in
H$, so $s_{s(\mu)} \in I$. Thus $s_\mu = s_\mu s_\mu^* s_\mu =
s_\mu s_{s(\mu)} \in I$. So $I$ contains all the generators of
$C^*(\Lambda)$, forcing $I = C^*(\Lambda)$.

$(3) \Rightarrow (2)$: We prove the contrapositive statement.
Suppose statement (2) does not hold. Fix $x \in \partial
\Lambda$ and $v \in \Lambda^0$ such that $v \Lambda x(n) =
\emptyset$ for all $n \in \mathbb{N}^k$. Define $T \subseteq
\partial \Lambda$ by
\[
T = \{ y \in \partial \Lambda : \sigma^m (y) = \sigma^n (x) ~ \mathrm{for ~ some} ~ m,n \in \mathbb{N}^k \}.
\]

Let $W = \clsp \{ \zeta_y : y \in T \} \subseteq \ell^2
(\partial \Lambda)$ and fix $y \in T$. Since $\lambda y,
\sigma^n (y) \in T$ for all $\lambda \in \Lambda r(y)$ and $n
\leq d(y)$, we have $S_\lambda W \subseteq W$ and $S_\lambda^*
W \subseteq W$ for every $\lambda \in \Lambda$. Since the
$S_\lambda$ satisfy the Cuntz-Krieger relations, so do the
$S_\lambda |_W$.

By the universal property of $C^*(\Lambda)$ there is a
representation $\pi : C^*(\Lambda) \rightarrow \mathcal{B}(W)$
such that $\pi (s_\lambda) = S_\lambda$ for every $\lambda \in
\Lambda$. We claim that $I := \ker(\pi)$ contains a vertex
projection but is not equal to $C^*(\Lambda)$. First observe
that $\pi(s_{r(x)}) \zeta_x = S_{r(x)} \zeta_x = \zeta_x$, so
$I \neq C^*(\Lambda)$. We will show that $s_v \in I$. To see
this, fix a basis element $\zeta_y \in W$. Since $y \in T$ we
have $\sigma^m (y) = \sigma^n (x)$ for some $m, n \in \N^k$. In
particular $y(0, m) \in r(y) \Lambda x(n)$ but $v \Lambda x(n)
= \emptyset$ which forces $r(y) \neq v$. Hence $\pi(s_v)
\zeta_y = 0$.

$(3) \Rightarrow (4)$: For $E \subseteq \Lambda$ define $M_{\Pi
E}^s := \clsp \{ s_\mu s_\nu^* : \mu, \nu \in \Pi E, d(\mu) =
d(\nu) \}$. Then the proof of \cite[Theorem~3.1]{rsy2} shows
that $C^*(\Lambda)^\gamma = \overline{\bigcup_{E \subset
\Lambda ~ \mathrm{finite}} M_{\Pi E}^s}$, and
\cite[Lemma~3.2]{rsy2} implies that each $M^s_{\Pi E}$ is a
finite-dimensional $C^*$-algebra. Let $I$ be an ideal of
$C^*(\Lambda)$ such that $I \cap C^*(\Lambda)^\gamma \neq \{ 0
\}$. Then there exists a finite set $E \subseteq \Lambda$ such
that $I \cap M_{\Pi E}^s \neq \{ 0 \}$. Since $M_{\Pi E}^s =
\oplus_{v \in s(\Pi E)} M_{(\Pi E)v}^s$ and $M_{(\Pi E)v}^s$ is
simple for each $v \in s(\Pi E)$, there exists $v \in s(\Pi E)$
such that $I \cap M_{(\Pi E)v}^s = M_{(\Pi E)v}^s$. Since $(\Pi
E)v \neq \emptyset$ there exists $\lambda$ such that $s_\lambda
s_\lambda^* \in M_{(\Pi E)v}^s \subset I$. Hence
$s_{s(\lambda)} = s_\lambda^* (s_\lambda s_\lambda^*) s_\lambda
\in I$ and $(3)$ implies $I = C^*(\Lambda)$.

$(4) \Rightarrow (3)$: Trivial, since each $s_v$ belongs to
$C^*(\Lambda)^\gamma$.
\end{proof}

\begin{proof}[Proof of Theorem~\ref{main}]
($\Rightarrow$). Suppose $C^*(\Lambda)$ is simple. Then $(3)
\Rightarrow (1)$ of Theorem~\ref{ideal} implies that $\Lambda$
is cofinal and $(2) \Rightarrow (1)$ of Theorem~\ref{4.3rob-s2}
implies that $\Lambda$ is aperiodic.

($\Leftarrow$). Suppose $\Lambda$ is aperiodic and cofinal. Fix
a nonzero ideal $I$ of $C^*(\Lambda)$. Then $(1) \Rightarrow
(2)$ of Theorem~\ref{4.3rob-s2} implies that $I$ contains a
vertex projection, and $(1) \Rightarrow (3)$ of
Theorem~\ref{ideal} then implies $I = C^*(\Lambda)$. Thus
$C^*(\Lambda)$ is simple.
\end{proof}

\appendix
\section{Cofinality}\label{lg}

In the appendix we will show that our cofinality condition for
finitely aligned $k$-graphs is equivalent to other, simpler
conditions for less general classes of $k$-graphs.

\subsection{Row-finite locally convex $k$-graphs}
We begin with some notation: for $n \in \N^k$ we write $|n| :=
\sum_{i = 1}^k n_i \in \N$. To show the ``if'' direction for
row-finite locally convex $k$-graphs we will need to use the
following technical lemma.

\begin{lemma}\label{F}
Let $\Lambda$ be a locally convex $k$-graph and fix $\lambda\in
\Lambda^{\leq m + n}$. Then $\mu = \lambda \big( 0, m \wedge
d(\lambda) \big)$ and $\nu = \lambda \big( m \wedge d(\lambda),
d(\lambda) \big)$ are the unique paths $\mu \in \Lambda^{\leq
m}$ and $\nu \in  \Lambda^{\leq n}$ such that $\lambda = \mu
\nu$.
\end{lemma}
\begin{proof}
We first establish the existence of paths $\mu \in
\Lambda^{\leq m}$ and $\nu \in \Lambda^{\leq n}$ such that
$\lambda = \mu \nu$ and then show that they must be as defined
above. We proceed by induction on $|n|$. If $|n|= 1$ then this
is precisely \cite[Lemma~3.12]{rsy1}. Suppose the statement is
true for $|n| = l \geq 1$ and suppose $|n| = l + 1$. Fix $j \in
\{ 1,\ldots,k \}$ such that $n_j \geq 1$. By
\cite[Lemma~3.12]{rsy1} we may factorise $\lambda$ as $\lambda'
\lambda''$ where $\lambda' \in \Lambda^{(m + n) - e_j}$ and
$\lambda'' \in \Lambda^{\leq e_j}$. The inductive hypothesis
implies that $\lambda' = \mu \lambda'''$ where $\mu \in
\Lambda^{\leq m}$ and $\lambda''' \in \Lambda^{\leq n - e_j}$.
By \cite[Lemma~3.6]{rsy1} we have $\nu := \lambda''' \lambda''
\in \Lambda^{\leq (n - e_j) + e_j} = \Lambda^{\leq n}$ and
$\lambda = \mu \nu$ as desired.

For uniqueness first observe that if $p \leq d(\lambda)$
satisfies $\lambda(0, p) \in \Lambda^{\leq m}$, then $p \leq m
\wedge d(\lambda)$. Now suppose for contradiction that $p_i <
\big( m \wedge d(\lambda) \big)_i$. Then $p_i < m_i$ so $p +
e_i \leq d(\lambda)$; hence $\lambda(p, p + e_i) \in \lambda(p)
\Lambda^{e_i}$ contradicting $\lambda(0, p) \in \Lambda^{\leq
m}$. It follows that if $\lambda = \mu \nu$ with $\mu \in
\Lambda^{\leq m}$, we must have $\mu = \lambda \big( 0, m
\wedge d(\lambda) \big)$, and then $\nu = \lambda \big( d(\mu),
d(\lambda) \big)$ by the factorisation property.
\end{proof}

\begin{prop}
Suppose $\Lambda$ is a row-finite locally convex $k$-graph.
Then $\Lambda$ is cofinal if and only if for every $v, w \in
\Lambda^0$ there exists $n \in \N^k$ such that $v \Lambda
s(\alpha) \neq \emptyset$ for each $\alpha \in w \Lambda^{\leq
n}$.
\end{prop}
\begin{proof}
$(\Rightarrow)$: Suppose $\Lambda$ is cofinal. Fix $v, w \in
\Lambda^0$. Then there exists $E \in w\FE(\Lambda)$ such that
$v \Lambda s(\alpha) \neq \emptyset$ for each $\alpha \in E$.
We must show there exists $n \in \N^k$ such that $v \Lambda
s(\beta) \neq \emptyset$ for each $\beta \in w \Lambda^{\leq
n}$. Let $n = \bigvee_{\alpha \in E} d(\alpha)$ and fix $\beta
\in w\Lambda^{\leq n}$. Since $E$ is exhaustive there exists
$\alpha \in E$ such that $\MCE(\alpha, \beta) = \alpha \mu =
\beta \nu$ for some $\mu \in s(\alpha) \Lambda^{(d(\alpha) \vee
d(\beta)) - d(\beta)}$ and $\nu \in s(\alpha)
\Lambda^{(d(\alpha) \vee d(\beta)) - d(\alpha)}$. However since
$\beta \in \Lambda^{\leq n}$ we must have $\nu = s(\beta)$.
Hence $v \Lambda s(\alpha) \neq \emptyset$ and thus $v \Lambda
s(\beta) \supseteq v \Lambda s(\alpha) \mu \neq \emptyset$.

$(\Leftarrow)$: it suffices to show that each $w \Lambda^{\leq
n} \in w\FE(\Lambda)$. Since $\Lambda$ is row-finite we have
$|w \Lambda^{\leq n}| < \infty$. Fix $\alpha \in w \Lambda$ and
$\beta \in s(\alpha) \Lambda^{\leq (n \vee d(\alpha)) -
d(\alpha)}$. By \cite[Lemma~3.6]{rsy1} $\alpha \beta \in w
\Lambda^{\leq n \vee d(\alpha)} = w \Lambda^{\leq n + ((n \vee
d(\alpha)) - n)}$. Then by Lemma~\ref{F} there exists $\mu \in
\Lambda^{\leq n}$ and $\nu \in w \Lambda^{\leq (n \vee
d(\alpha)) - n}$ such that $\mu \nu = \alpha \beta$. Hence $\mu
\in \Lambda^{\leq n}$ satisfies $\MCE(\mu, \alpha) \neq
\emptyset$.
\end{proof}

\begin{remark}
If $\Lambda$ has no sources, then each $\Lambda^{\leq n} =
\Lambda^n$, so $\Lambda$ is cofinal if and only if for every
$v, w \in \Lambda^0$ there exists $n \in \N^k$ such that $v
\Lambda s(\alpha) \neq \emptyset$ for every $\alpha \in w
\Lambda^n$.
\end{remark}

\subsection{Arbitrary directed graphs}

We will use the following notation to characterise cofinality
in an arbitrary directed graph.

\begin{notation*}
Suppose $E$ is an arbitrary directed graph. For $n \in \N$ we
define
\begin{align*}
X_n := \big\{\lambda \in E^* : d(\lambda) \leq n, & |\lambda(m) E^1| < \infty \text{ for every } m < d(\lambda) \\
    & \text{and if } d(\lambda) < n \text{ then } |s(\lambda) E^1| \in \{ 0, \infty \} \big\}.
\end{align*}
\end{notation*}

\begin{prop}
Suppose $E$ is an arbitrary directed graph. Then $\Lambda$ is
cofinal if and only if for every $v, w \in E^0$ there exists $n
\in \N$ such that $v E^* s(\alpha) \neq \emptyset$ for each
$\alpha \in w X_n$.
\end{prop}
\begin{proof}
$(\Rightarrow)$: Suppose $\Lambda$ is cofinal. Fix $v, w \in
\Lambda^0$. Then there exists $F \in w\FE(E)$ such that $v E^*
s(\alpha) \neq \emptyset$ for every $\alpha \in F$. Let $n =
\max \{ |\alpha| : \alpha \in F \}$. Consider $w (X_n)$. We
show that $v E^* s(\mu) \neq \emptyset$ for every $\mu \in w
(X_n)$. Fix $\mu \in w (X_n)$. Then there exists a smallest
$\nu \in F$ such that $\MCE(\mu, \nu) \neq \emptyset$.

We claim that $|\mu| \geq |\nu|$. Suppose for contradiction
that $|\mu| < |\nu|$. Then $|\nu (m) E^1| < \infty$ for every
$m < |\nu|$; for if $m < |\nu|$ such that $|\nu (m) E^1| =
\infty$ then there exists $\lambda \in \nu (m) E^1$ such that
$\nu(0, |m|) \lambda \notin F$, which is impossible because
$\nu$ is minimal in $F$ with $\MCE(\mu, \nu) \neq \emptyset$.
In particular $|s(\mu) E^1| < \infty$. However, because $|\mu|
< |\nu| \leq n$, by definition of $X_n$ either $s(\mu) E^1$ is
either empty or infinite. Clearly $|s(\mu) E^1| \neq 0$ so
$|s(\mu) E^1| = \infty$ giving a contradiction. Now since
$|\mu| \geq |\nu|$ we have $\MCE(\mu, \nu) = \mu = \nu
\mu(|\nu|, |\mu|)$ and it follows from $v E^* s(\nu) \neq
\emptyset$ that $v E^* s(\mu) \neq \emptyset$.

$(\Leftarrow)$: Fix $v, w \in E^0$. Then there exists $n \in
\N$ such that $v E^* s(\alpha) \neq \emptyset$ for each $\alpha
\in w (X_n)$. By definition of $X_n$ we have $|X_n| < \infty$.
We claim that $w(X_n)$ is exhaustive. Suppose for contradiction
that $w (X_n)$ is not exhaustive. Then there exists $\lambda
\in wE^*$ such that $\MCE(\beta, \lambda) = \emptyset$ for
every $\beta \in w (X_n)$. Since $r(\lambda) = w$ there exists
a smallest $m \leq n$ such that $\lambda(0, m) = \beta(0, m)$
for some $\beta \in w X_n$ but $\lambda(m, m + 1) \neq \mu
\big(m, (m + 1) \wedge d(\mu) \big)$ for every $\mu \in w
(X_n)$. We consider two cases: when $m = n$ and when $m < n$.

We first consider the case when $m = n$. Then $\beta(m) =
s(\beta)$ and since $\beta(m) E^* s(\lambda) \neq \emptyset$ it
follows that $\MCE(\beta, \lambda) = \lambda$. Now suppose $m <
n$. Then $m + 1 \leq n$ and hence by definition of $X_n$ we
have $|\beta(m) E^1| = \infty$ and hence $|\beta| = m$ and
again it follows that $\MCE(\beta, \lambda) = \lambda$. Hence
$w X_n$ is exhaustive.
\end{proof}

\end{document}